\documentclass{proc-l}

\usepackage{amssymb}

\newtheorem{theorem}{Theorem}[section]
\newtheorem{lemma}[theorem]{Lemma}

\theoremstyle{definition}

\newtheorem{question}[theorem]{Question}

\newtheorem{corollary}[theorem]{Corollary}
\newtheorem{fact}[theorem]{Fact}
\newtheorem{proposition}[theorem]{Proposition}
\newtheorem{claim}[theorem]{Claim}
\theoremstyle{remark}

\numberwithin{equation}{section}

\def\Z{{\mathbb Z}}

\def\Zp{{\Z}_p}

\newcommand{\R}{{\mathbb R}}
\newcommand{\N}{{\mathbb N}}
\newcommand{\C}{{\mathbb C}}
\newcommand{\K}{{\mathbb K}}
\newcommand{\HH}{{\mathbb H}}
\newcommand{\Q}{{\mathbb Q}}
\newcommand{\Prm}{{\mathbb P}}
\newcommand{\T}{{\mathbb T}}
\newcommand{\sq}{\subseteq}

\newcommand{\cc}{countably\ compact}
\newcommand{\mi}{minimal}

\newcommand{\tm}{totally minimal}
\newcommand{\co}{connected}
\newcommand{\sco}{sequentially complete} 
\def\td{totally\ disconnected}

\begin{document}

 \title[Countably compact groups with minimal powers]{Countably compact groups having minimal infinite powers}


\author{Dikran Dikranjan}
\address{Dipartimento di Scienze Matematiche, Informatiche e Fisiche\\
Universit\`{a} di Udine}
\curraddr{}
\email{dikranja@dimi.uniud.it}

\author{Vladimir Uspenskij}
\address{Department of Mathematics, Ohio University, Athens, OH 45701, USA}
\curraddr{}
\email{uspenski@ohio.edu}
\thanks{}

\thanks{}


\date{}

\dedicatory{}

\commby{Vera Fischer}

\begin{abstract}
We answer the question, raised more than thirty years ago in \cite{DS2,DS4}, on whether the power $G^\omega$ of a \cc\ minimal  Abelian group  $G$ is minimal, by showing that the negative answer is equivalent to the existence of measurable cardinals. 
The proof is carried out in the larger class of \sco \ groups. We characterize the \sco \ minimal  Abelian groups $G$ such that $G^\omega$ is minimal -- these are
exactly those $G$ that contain the connected component of their completion. 
This naturally leads to the next step, namely, a better understanding the structure of the \sco \ minimal Abelian groups,
and in particular, their connected components which turns out to depend of the existence of   Ulam measurable cardinals. More specifically, 
all connected \sco \ minimal Abelian groups are compact,  if Ulam measurable cardinals do not exist. On the other hand, for every  Ulam measurable
cardinal $\sigma$ we build  a non-compact torsion-free connected
\mi\ $\omega$-bounded Abelian group of weight $\sigma$, thereby showing that the Ulam measurable cardinals are precisely the weights
of non-compact \sco\ connected \mi\ Abelian groups.
\end{abstract}

\maketitle

\section{Introduction}

We denote by ${\N}$ set of naturals $\{0,1,\dots\}$ (and $\N^*=\N\setminus\{0\}=\{1,2,\dots\}$), 
by ${\Prm}$ the set all primes,  by ${\Z}$ the integers, by ${\Q}$  the rationals, by  ${\R}$ the reals,  by ${\T}$  the unit circle group in ${\C}$,  by $\Z_p$ the $p$-adic integers ($p\in \Prm$), by $\Z(n)$ the cyclic group of order $n$ ($n\in \N^*$).  The cardinality of continuum $2^\omega$ will be denoted also by ${\mathfrak c}$.

All groups considered in this paper (except in \S 3) will be Abelian, so additive notation will be always used. 
We denote by $\widetilde G$ the completion of a topological group $G$ (we do not have to worry about left, right, and two-sided uniformities, since our groups are Abelian),
while $c(G)$ denotes the \co \ component of the neutral element of a group $G$ (for brevity, we call $c(G)$ the connected component of $G$). 

\subsection{Minimal groups and their products}

Following Stephenson \cite{St}, call a Hausdorff topological group $(G,\tau)$ {\it minimal} if the topology $\tau$ is a  minimal element of the  partially ordered, with respect to inclusion, set of Hausdorff group topologies on $G$.  Compact groups are obviously \mi, the  first examples of non-compact \mi \ groups were given in \cite{St,Do1}, while  the  first examples of non-compact countably compact \mi \ groups  were given in \cite{CG} (see also \cite{DS2}--\cite{DT1}).

The problem of preservation of minimality under products (the counterpart of Tychonoff product theorem for compactness) was raised by
G. Choquet at the ICM in Nice in 1970. The first counter-example was given in \cite{Do1}. Since then this problem became a central one in the 
framework of minimal groups \cite{Do1,D5,DPS,DS1,DS4,DTo,EDS,S2,S4}.  A general criterion for minimality of arbitrary products of minimal Abelian groups
was produced in \cite{D5,DS1} (see also \cite[Theorem 2.4]{DS1} and \cite[Theorem 6.2.3]{DPS}). The minimality of arbitrary powers of  an  Abelian group $G$ was characterized earlier by Stoyanov \cite{S4}, who proved that it is equivalent to the minimality of $G^{\mathfrak c}$ (see Theorem \ref{Stoy} for a more detailed description). Since minimality is preserved by taking direct summands, the minimality of some power $G^\kappa$ implies the minimality of all smaller powers. 

It is natural to expect that adding some additional compact-like properties may improve productivity of minimal groups.
Indeed,  all finite powers of a \cc \  minimal group are minimal \cite{DS4}. On the other hand, there exist pseudocompact minimal Abelian groups $G$ with  $G^{\omega_1}$ non-\mi, but $G^\omega$ minimal (for more detail see \cite{D8,DS4}). 

The choice of compact-like property  to match with minimality in this paper 
is sequential completeness, since it simultaneously generalizes two fundamental properties like   
countable compactness and completeness. A topological group $G$ is {\em \sco} if $G$ is sequentially closed in 
$\widetilde G$ (or, equivalently, when every Cauchy sequence in $G$ converges) \cite{DT1,DT3}. 
We prove that all finite powers of a
\sco\  minimal group are minimal (this is a corollary of a more general fact, see Lemma \ref{cc_perf_min}).  This suggests  the following question, set in \cite[Question 9]{DS2} and \cite[Question 1.10]{DS4} in the case of \cc\ groups which still remains open since then:

\begin{question}\label{Ques_Nico}  Let $G$ be a \sco\  minimal Abelian group. Is $G^\omega$ minimal ? \end{question}

We shall see below that this question, as well as the question in  \cite{DS4} regarding \cc\ groups, cannot be  answered in ZFC.

In order to answer Question \ref{Ques_Nico}, we characterize first the \sco \  Abelian groups $G$ such that $G^\omega$ is minimal (i.e., those providing a positive answer to Question \ref{Ques_Nico}): 

\bigskip

\noindent {\bf Theorem A.}  {\em For a  \sco \ minimal  Abelian group $G$ the following are equivalent:
\begin{itemize}
         \item[(a)]  all powers of $G$ are minimal;
         \item[(b)] $G^\omega$ is minimal.  
         \item[(c)] $G$ contains $c(\widetilde{G})$;  hence, $c(G)=c(\widetilde{G})$ is compact. 
\end{itemize} }

\medskip

The proof of this theorem is given in  \S \ref{proofs}, where we prove it first in the totally disconnected case, when $G^\omega$ is always minimal (Corollary \ref{beta(td)}),
i.e., Question \ref{Ques_Nico} has positive answer for totally disconnected $G$. The equivalence of (b) and a weaker form of (c) was established in \cite[Main Theorem]{DTo} in the much smaller class of $\omega$-bounded \mi \ Abelian groups. 

\smallskip
On the other hand, we shall see in Corollary \ref{td_omega} that the assertion ``all powers of a  \cc\  minimal  Abelian groups are \mi" is equivalent to the non-existence of
Ulam-measurable cardinals. Let us recall that a cardinal $|X|$ is {\it Ulam-measurable} 
(resp., {\it measurable}) if there exists an $\omega_1$-complete (resp., $|X|$-complete) free ultrafilter on $X$ (a filter on a set $X$ is $\kappa$-complete if it is closed under intersections of families of cardinality $<\kappa$).  A cardinal is Ulam-measurable if and only if it greater than or equal to the first measurable cardinal. 
The assumption that there exist no Ulam-measurable cardinals is  known to be consistent with ZFC, while 
their existence implies consistency of ZFC and hence no proof of that existence is to be expected  \cite{J}. 

Uncountable measurable cardinals are Ulam-measurable, while the
least Ulam-measurable cardinal is also measurable. 

For the connected case Theorem A gives the following striking consequence characterizing the minimality of $ G^\omega$ by compactness of $G$:

\begin{corollary}\label{cc_kappa1} A connected  \sco\  \mi \ Abelian group  $G$ is compact if and only if $G^\omega$ is minimal. \end{corollary}

This corollary shows that in order to positively answer Question \ref{Ques_Nico} in the connected case, we have to
check whether the connected  \sco\  \mi \ Abelian group  $G$ are compact. This issue turns out to be non--trivial, we shall face it in \S \ref{Structure}
 which mainly deals with this aspect of the structure of  \sco \ minimal Abelian groups (see Question \ref{Ques_U}). 

Here is another application of Theorem A. The general criterion for  minimality of arbitrary products of  Abelian groups from \cite[Theorem 6.2.3]{DPS},
simplifies in the case of \sco \ group (due to Lemma \ref{cc_perf_min}), so that applying the criterion from \cite[Corollary 6.2.8]{DPS} one gets the following
corollary (for a proof see \S 2.2).

\begin{corollary}\label{Crit_prod} Let $\{G_i\}_{i\in I}$ be a family of \sco \ Abelian groups.  Then $G=\prod_{i\in I}G_i$  is minimal if and only if  every countable subproduct is minimal. In such a case there exists a co-finite subset $J\subseteq I$ such that all groups $G_i^\omega$, $i\in J$, are \mi. \end{corollary}

Corollary \ref{cc_kappa1}  implies that if the groups $G_i$ in the above theorem are also \co, then $G$ is minimal if and only if  all but finitely many of the groups $G_i$ are compact. 

According to \cite[Theorem 1.5]{DS4} the last part of Corollary \ref{Crit_prod} cannot be inverted: there exists a family $\{G_n:n\in \N\}$ of $\omega$-bounded Abelian groups such that all powers $G_n^\lambda$ are minimal, but $\prod_{n=0}^\infty G_n$ is not minimal. 


\subsection{The structure of  \sco \ minimal groups}\label{Structure}

Ulam-measur\-able cardinals already appeared in the context of comparison between compactness and countable compactness \cite{A2}, \cite{CRe}, \cite{U2}.
Relaxing countable compactness to sequential completeness and adding minimality to the mix we consider the following

\begin{question}\label{Ques_U} Is every connected  \sco\ \mi \ Abelian group $G$ compact?  
\end{question}

Connectedness cannot be omitted here, since there exist \mi\ $\omega$-bounded proper dense subgroups of $\Z_p^{\omega_1}$  \cite{DSpams}
(they are obviously totally disconnected).

According to Corollary \ref{cc_kappa1} a positive answer to Question \ref{Ques_Nico} is equivaent to a positive answer to Question \ref{Ques_U}.

It turns out that the answer to Question \ref{Ques_U} depends on the existence of measurable cardinals, hence  Question \ref{Ques_U} and Question \ref{Ques_Nico} cannot be answered in ZFC. 

\medskip
\medskip

\noindent {\bf Theorem B.}
{\em Let $\alpha\geq \omega$ be a cardinal. Then the following conditions are equivalent: 
\begin{itemize}
        \item[(a)] $\alpha$ is Ulam-measurable;
        \item[(b)]  there exists a non-compact \co, sequentially complete,  minimal  Abelian group $G$ of weight $\alpha$;
        \item[(c)] there exists a non-compact \co,  $\omega$-bounded, minimal torsion-free  Abelian group $G$ of weight $\alpha$.
\end{itemize} 

Moreover, if $\alpha$  is the least  Ulam-measurable cardinal, then there exists a \co\ non-compact minimal torsion-free  Abelian group $G$ of weight $\alpha$
that is $\beta$-bounded for every $\beta<\alpha$.}

\medskip
\medskip

The proof of this  theorem is given in \S \ref{proofs}. Since the  least  Ulam-measurable cardinal is a measurable cardinal, we prove there a sightly more general version of the last assertion of the theorem. Namely, for every measurable cardinal $\alpha$ there exists a \co\ non-compact minimal torsion-free  Abelian group $G$ of weight $\alpha$ that is $\beta$-bounded for every $\beta<\alpha$. It should be noted here that  the last property is very close to compactness, since a space $X$ that is $w(X)$-bounded must be compact. 
 
 On the other hand, the implication (b) $\to$ (a) of the above theorem follows from the fact that a connected  \sco\ \mi\ Abelian group of non-measurable size is compact due to item (b) of the following: 
    
   \bigskip

\noindent {\bf Theorem C.}  {\em Let $G$ be a \sco\  minimal  Abelian group. 
\begin{itemize}
         \item[(a)] If $c(G)$ is compact, then $c(G)=c(\widetilde{G})$ is compact and $G/c(G)$ is \sco  \ and minimal. 
         \item[(b)]  If $w(c(G))$ is not Ulam-measurable, then $c(G)$ is compact. 
\end{itemize}
}

 \bigskip
 
In particular, a \sco\  Abelian group $G$ with non Ulam-measurable $w(c(G))$ is \mi\  if and only if $c(G)$ is compact and $G/c(G)$ is \mi \ (the sufficiency uses the three space property, see Proposition \ref{3space}, or \cite{EDS}).  This reduces the study of \cc\ \mi \  Abelian groups with small  \co \ component to that of \td\ groups.

Now we come back to Question \ref{Ques_Nico}. First we give a sufficient condition for a positive answer:

\begin{corollary}\label{DT} Let $G$ be a \sco \ \mi\ \  Abelian group such that $w(c(G))$ is not Ulam measurable. Then $G^\omega$ is \mi. 
 \end{corollary}

Indeed, the hypotheses of the corollary imply that $c(G)=c(\widetilde G)$ is compact (in view of Theorem C(b)), so $G^\omega$ is \mi, by Theorem A. This proves item (a) of the next corollary, in view of Theorem B. Actually, the next corollary shows that Question \ref{Ques_Nico} cannot be answered neither for \cc \ groups, nor for \sco\ groups. 

 \begin{corollary}\label{td_omega} (a) Under the assumption that there exist no Ulam-measurable cardinals, all powers of a \sco \ minimal  Abelian group are \mi, i.e., 
  Question \ref{Ques_Nico} has a positive answer.

(b) Under the assumption that there exist Ulam-measurable cardinals, there exists an $\omega$-bounded 
minimal Abelian group $G$ such that $G^\omega$ is not \mi,   Question \ref{Ques_Nico} has a negaitive answer.  \end{corollary}

Item (b) follows from item (b) of Theorem B, which ensures the existence of a   non-compact \co,   sequentially complete,  minimal  Abelian group $G$ of weight $\alpha$, whenever $\alpha$ is an Ulam-measurable cardinal. Then  $G^\omega$ is not \mi, according to Corollary \ref{cc_kappa1}. 

Finally, we offer a description of the \sco\  Abelian group $G$ such that $G^\omega$ is \mi, i.e., those satisfying the equivalent conditions of  Theorem A.
They turn out to be extensions of a compact group by a minimal Abelian group that is a direct product  of bounded minimal $p$-groups.   

\medskip
\medskip
 \noindent {\bf Theorem D.} {\em  Let $G$ be a \sco\  Abelian group such that $G^\omega$ is \mi. Then there exists a compact subgroup $N$ of $G$ such that $G/N\cong \prod_{p \in \Prm} B_p$, where $B_p$ is a bounded $p$-torsion minimal group for every prime $p$. }

\medskip This theorem is proved in  \S \ref{proofs}.  The subgroup $N$ is ``large", it contains the compact subgroup $c(\widetilde G)$ and much more
(for a stronger and more precise form of the theorem see Theorem \ref{TD_cor}).

\medskip 

The paper is organzied as follows. In \S \ref{BackMi} we provide background on minimality and sequential completeness. In particular, 
we recall two criteria for minimality of dense subgroups of compact Abelian groups and we give a factorization theorem for totally disconnected  \sco\ \mi \ Abelian group (Theorem \ref{Products}). Section \ref{proofs} contains the proofs of Theorems A, B, C, D and Corollary \ref{Crit_prod}. A key point of \S \ref{proofs} is Lemma \ref{coroVU}, which is the clue to the proofs of Theorems A, C and D. 

\subsection*{Notation and terminology}
Let $G$  be a group and  $A$  be a subset of $G$.   
We denote  by $\langle A\rangle$ the subgroup of $G$ generated by $A$. The group $G$ is {\it divisible}  if for every $g\in G$ and $n\in \N^*$ the equation $nx=g$ has a solution in $G$. For an  Abelian group $G$ we  set $G[n] = \{x \in  G: nx=0\}$, for $n\in \N^*$, and $Soc(G) = \bigoplus_{p\in \Prm}G[p]$.  

We recall some notions of compactness-like conditions in topological groups and spaces. A Tychonoff space $X$ is {\it pseudocompact} if every continuous real-valued function on $G$ is bounded, and {\it countably compact} if every countable open cover has a finite subcover (equivalently, every
sequence in $X$ has a cluster point).
For an infinite cardinal $\alpha$, a group $G$ is {\it $\alpha$-bounded } if every subset of cardinality  $\leq
\alpha$  of $G$  is contained in a compact subgroup of $G$ (some authors use the term $\omega$-bounded for a different property). 
 Obviously, $\omega$-boundedness implies countable compactness, and countable compactness  implies pseudocompactness.

For undefined symbols or notions see \cite{DPS,E,Fuchs}.

\section{Proofs of the main results}\label{Ulam+U} 

\subsection{Background on minimality and sequential completeness}\label{BackMi}

 A topological group  $G$ is {\it precompact} if its  completion $\widetilde G$ is compact 
(or, equivalently, if for any open $U\ne \emptyset$ in $G$ there is a finite subset $F\sq G$ such that $F+U=G$). Peudocompact groups are
precompact. The following fundamental theorem of Prodanov and Stoyanov \cite{PS2} (see also \cite[Theorem 2.7.7]{DPS} for an alternative 
proof) says that the \mi \ Abelian groups are precisely the (dense) subgroups of the compact Abelian groups: 

\begin{theorem}\label{Prec:Thm} {\rm \cite{PS2}} Every \mi\  Abelian group is precompact.  
 \end{theorem}

A subgroup $H$ of a topological Abelian group $G$ is {\it essential} if  every non-trivial closed subgroup $N$ of $G$ meets $H$ non-trivially. 
The following criterion for   \mi ity of dense subgroups, given by Stephenson \cite{St} and Prodanov \cite{Prod}, describes the minimal Abelian groups as the  dense essential subgroups of the compact Abelian groups: 

\begin{theorem}\label{Tot_Min_Crit}  Let $G$ be a compact Abelian group and $H$ be a dense subgroup of $G$. Then $H$ is minimal if and only if $H$ is essential in $G$.
 \end{theorem}

Actually, it is possible to carry out the test of essentiality with only closed subgroups $N$ of $G$ that are either cyclic $p$-group or copies of $\Z$, with $p\in \Prm$ :

 \begin{lemma}\label{loc:criterion}{\rm \cite[Theorem 4.3.7]{DPS}} Let $G$ be a compact Abelian group. A dense subgroup $H$ of $G$ is essential if and only if  $Soc(G) \leq H$ and for every prime $p$ the subgroup $H$ non-trivially meets every subgroup $N\leq G$ with $N \cong \Z_p$.   
  \end{lemma}

The reduction in the above lemma is based on a the localization technique invented by Stoyanov \cite{S2} (see also \cite[Chapter 4]{DPS}) that we recall now. 
 For a prime $p$ an element $x$ of a topological Abelian group $G$ is called  {\it quasi-$p$-torsion} if either $x$ is $p$-torsion or $\langle x\rangle$ is isomorphic to $\Z$ equipped with the $p$-adic topology. The set $td_p(G)$ of all quasi-$p$-torsion elements of $G$ is a subgroup of $G$. In these terms, Lemma \ref{loc:criterion} ensures that $H$ is essential in $G$ if and only if  $td_p(H)$ is essential in $td_p(G)$ for every $p$ (see \cite[Theorem 4.3.7]{DPS}). 

Here we collect some properties of the subgroup $td_p(-)$ that will be used in the sequel, for further properties see \cite[Chapter 4]{DPS}, 

 \begin{fact}\label{prop_td_p} {\rm \cite{DPS}} 
 \begin{itemize}
\item[(1)] Let $f: H \to G$ be a continuous homomorphisms of  topological Abelian groups.  Then for every prime $p$: 
   \begin{itemize}
     \item[(1a)] $f(td_p(H))\leq td_p(G)$, if $H$ is compact and $f$ is surjective, then \\ $f(td_p(H)) = td_p(G)$; 
     \item[(1b)] if $f$ is an embedding, then $f(td_p(H))= td_p(G) \cap f(H)$.
\end{itemize}
    \item[(2)] If $\{G_i:i\in I\}$ is a family of topological Abelian groups, then $td_p(\prod_{i\in I}G_i) = \prod_{i\in I} td_p(G_i)$.
    \item[(3)] If $G$ is a totally disconnected compact Abelian group, then  for every $p\in \Prm$ the subgroup $td_p(G)$ of $G$  is closed and 
$G\cong \prod _{p\in \Prm} td_p(G)$ topologically. 
\end{itemize}
  \end{fact}

Following \cite{S2}, call a topological Abelian group  $G$ {\em strongly $p$-dense} if there exists $k\in \N^*$ such that $p^ktd_p(\widetilde{G})\subseteq G$.
This technical property turns out to be the key towards complete understanding of minimality of powers: 

\begin{theorem}\label{Stoy} {\rm \cite[Theorem 3.5]{S2}} For a minimal  Abelian group $G$ the following are equivalent: 
\begin{itemize}
         \item[(a)] all powers of $G$ are minimal;
         \item[(b)] $G^{\mathfrak c}$ is minimal;
         \item[(c)] $G$ is strongly $p$-dense for every prime $p$. 
\end{itemize}        \end{theorem}

For the proof of Theorem A we need the following lemma: 

\begin{lemma}\label{Str_d} If a \sco \   Abelian group $G$ is strongly $p$-dense for some prime $p$, then $G\supseteq c(\widetilde{G})$. \end{lemma}

\begin{proof} In fact $c(\widetilde{G})$ is connected hence divisible,  therefore $td_p(c(\widetilde{G}))$ is divisible as well. Now strong $p$-density implies $td_p(c(\widetilde{G}))\subseteq G$. Since $td_p(c(\widetilde{G}))$ is sequentially dense in $c(\widetilde{G})$ by  \cite[Theorem 2.9(a)]{DT1},  it follows also that $G\cap c(\widetilde{G})$ is sequentially dense in $c(\widetilde{G})$. Now the sequential completeness of $G$ yields $G\supseteq c(\widetilde{G})$.  \end{proof}

The next lemma resolves the problem of minimality for finite products of sequentially complete \mi \ groups:

\begin{lemma}\label{cc_perf_min}  Let $G$ be a \sco \ minimal group. Then  $G\times H$ is \mi\ for every minimal group $H$.
In particular, any finite product of \sco \ minimal groups is minimal. 
\end{lemma}

\begin{proof} Fix an arbitrary prime $p$ and pick $x\in td_p(G)$. It suffices to check that $x$ is contained is a compact subgroup of $G$, then  
\cite[Proposition 6.1.13]{DPS} (see also \cite[Proposition 4.1]{DS1}) applies. The case when $C= \langle x \rangle$ is finite is trivial. Assume that $C$ is infinite. 
Then $C$ is metizable, being isomorphic to $\Z$ equipped with the $p$-adic topology. Hence its compact closure $\overline C$ in $\widetilde G$
is contained in $G$, since $G$ is \sco. 
\end{proof}

The \mi ity is not preserved by taking quotients (even quotients of minimal group with respect to finite subgroups may fail to be minimal). In item (b) of the next lemma
one can find an instance when  \mi ity is preserved by taking quotients. In item (a) we recall a ``three space property" from \cite{EDS}.

\begin{proposition}\label{3space}
Let $G$ be a topological Abelian group and let $K$ be a compact subgroup of $G$. Then
\begin{itemize}
\item[(a)] \cite{EDS} if $G/K$ is minimal, then $G$ is minimal; 
\item[(b)] if $G$ is minimal and $K$ is connected, then $G/K$ is minimal. 
\end{itemize}
\end{proposition}

\begin{proof} Let us only mention that (a) was proved in  \cite{EDS} without the assumption that $G$ is Abelian and compactness of $K$ is relaxed to 
the conjunction of minimality and completeness. 

(b) First we note that $G/K$ is precompact, by Theorem \ref{Prec:Thm}. Since $K$ is compact, this implies that $G$ is precompact too, 
so $\widetilde G$ is compact and $\widetilde G/K \cong \widetilde{G/K}$. 
In order to apply Lemma \ref{loc:criterion} and Theorem \ref{Prec:Thm} we need to check first that 
 $Soc(\widetilde{G/K}) \leq G/K$. Pick $a' = a+ K\in Soc(\widetilde{G/K})$. Then $pa\in K$. 
 Since $K$ is a connected compact Abelian group, $K$ is divisible. Hence, $pa = pb$, for some $b \in K$. 
 Then $t=a-b \in Soc(\widetilde{G}) \leq G$, by Lemma \ref{loc:criterion} and Theorem \ref{Prec:Thm}, since $G$ is minimal. 
Therefore, $a' = t + b + K= t+ K \in G/K$. In order to conclude with the application of Lemma \ref{loc:criterion}, consider
a subgroup $\Z_p\cong \N^* \leq \widetilde{G}/K$ and put $q(N_1) =N$, where $q: \widetilde{G}\to \widetilde{G}/K$ is
the canonical homomorphism. In order to see that the short exact sequence $0\to K \to L \to N \to 0$ splits,
consider  its dual short exact sequence $0\to \widehat N \to \widehat N_1 \to \widehat K \to 0$ which splits, due to the divisibility of $\widehat N \cong \Z(p^\infty)$. 
 Hence, $L \cong K \times N_1$, where $N_1\cong \Z_p$ is a closed subgroup of $L$ with $q(N_1) = N$, so clearly the restriction $q\restriction_{N_1} :N_1 \to N$ is an isomorphism.
Again by Lemma \ref{loc:criterion} and Theorem \ref{Prec:Thm}, applied to the 
minimal group  $G$, $N_1\cap G \ne \{0\}$.  Since, $q\restriction_{N_1}$ is an isomorphism, this implies $N\cap (G/K) \ne \{0\}$.  
\end{proof}

\begin{corollary}\label{coro:3space}
Let $G$ be a topological Abelian group such that $c(G)$ is compact. Then $G$ is minimal if and only if $G/c(G)$ is minimal. 
\end{corollary}

Now we prove that the property (3) from Fact \ref{prop_td_p} can be extended to \sco \ minimal groups:

\begin{theorem}\label{Products}
Let $G$ be a \sco\ \mi \ totally disconnected Abelian group. Then  for every $p\in \Prm$ the subgroup $td_p(G)$ of $G$  is closed and $G\cong \prod _{p\in \Prm} td_p(G)$ topologically.
\end{theorem}

\begin{proof} The completion $\widetilde G$ of $G$ is compact by Prodanov-Stoyanov's Theorem \ref{Prec:Thm}. Moreover, $\widetilde G$ is totally disconnected by 
Theorem C (its proof does not depend on the proof of the current theorem). Hence, $\widetilde G=\prod_pK_p$, where $K_p=td_p(\widetilde G)$ for every $p\in \Prm$, by Fact \ref{prop_td_p}(3).
Hence, for each $p\in \Prm$ the subgroup $td_p(G)=G\cap K_p$ of $G$ is closed, by items (1b) and (3) of Fact \ref{prop_td_p}. 
Since $G$ is \sco, to prove that $P:= \prod_{p\in \Prm}  td_p(G) \leq G$ it suffices to check that  the subgroup $S=\bigoplus  _{p\in \Prm} td_p(G)$ of  $G$ is sequentially dense in $P$. In fact, let $x=(x_p) \in P$. Then obviously $x\in \prod_{p\in \Prm} \langle x_p\rangle  $. As  this  subgroup of $P$ is metrizable, $x\in \overline{\bigoplus  _{p\in \Prm} \langle x_p \rangle }$ and $\bigoplus  _{p\in \Prm} \langle x_p \rangle  \sq S$,  
this proves the sequential density of $S$ in $P$.

\smallskip

In order to prove the desired inclusion $G \leq P$ it suffices to see that the canonical projection $\pi_p: \prod_{q\in \Prm}K_q\to K_p$ 
satisfies $\pi_p(G) \leq td_p(G)$, where we consider $\pi_p(G)$, as well as $K_p = \pi_p(K)$, as subgroups of the product $K$.
Since,  $td_p(G)=G\cap K_p$ and obviously $\pi_p(G) \leq K_p$, it suffices to check that $\pi_p(G)\leq G$,   

To  prove  this assertion take $x=(x_q)_{q\in \Prm} \in G$,  fix a prime $p$ and let $p_1, \ldots , p_n, \ldots  $ be an enumeration of the set $\Prm \setminus\{p\}$. 
Then we can consider $y=\pi_p(x)$ as the element of the product, such that $\pi_p (y)=x_p$, while  
  $\pi_{p_n} (y)= 0$  for each $n\in \N^*$. We have to prove that $y\in G$.
By the  Chinese  theorem  of remainders that there exists a sequence  $\{k_{n}\}$  in ${\Z}$   such that 
$$
k_n\equiv 1 \ (\hbox{mod }p^n)\  \ \hbox{ and }\ \ k_n\equiv 0 \ (\hbox{mod }(p_1\ldots p_n)^n)
$$ 
for each $n\in \N^*$. Then $k_n \rightarrow 1$  in the $p$-adic topology of $\Z$ and $k_n \rightarrow 0 $ in the $p_m$-adic topology 
of $\Z$ 
for each $m\in \N^*$. This yields 
\begin{equation}\label{CRT}
\lim_n k_n x_p = x_p  \ \ \mbox{ in }\ \  A_p \ \ \mbox{ and }\ \  \lim_n k_n x_{p_m} = 0  \ \ \mbox{ in }\ \ A_{p_m} \ \ \mbox{ for each }\ \  m\in \N^*
\end{equation}
since the topology of $A_q$ is coarser than the $q$-adic topology of $A_q$ for every $q \in \Prm$. 
 Since $G$ is \sco,    $\lim {k_n}x \in G$. On the other hand, (\ref{CRT}) entails $y = \lim {k_n}x \in G$. 
\end{proof}

We denote by ${\mathbb K}$ the (compact)  Pontryagin dual of the discrete group $\Q$, i.e., ${\mathbb K}$ is the group of all homomorphisms $\Q\to \T$ equipped with the topology of pointwise convergence. 
Since the Pontryagin dual of a compact torsion-free Abelian group is  a discrete divisible Abelian group, 
every compact torsion-free Abelian group is a product of copies of ${\mathbb K}$ and the groups $\Z_p$.
Hence, the next proposition reduces various problems related to \mi \ Abelian groups $G$ to the case of dense \mi \ subgroups of the powers ${\mathbb K}^\alpha$ or $\Z_p^\alpha$, where $\alpha = w(G)$.

\begin {proposition}\label{tor_free}  Let $G$ be a minimal  Abelian group. Then there exists a  compact torsion-free  Abelian 
group $K$ with $w(K)=w(G)$, a dense minimal  subgroup $G_1$ of $K$ and a compact totally disconnected subgroup $N$ of $G_1$ such that: 
\begin{itemize}
  \item[(a)] the quotient group $G_{1}/N$ is isomorphic to $G$ and consequently $K/N \cong  \widetilde {G}$;
  \item[(b)] for every $p\in \Prm$ $G$ is strongly $p$-dense if and only if $G_1$ is strongly $p$-dense; 
  \item[(c)] $G_1$ is \cc \ (resp., \sco, $\omega$-bounded,  \td) if $G$ is \cc \ (resp., \sco, $\omega$-bounded,  \td).    
\end{itemize} 
In case $G$ is a minimal quasi $p$-torsion group for some prime $p$, then $K=\Z_p^{w(G)}$.\end{proposition}

\begin{proof}
(a) and the part of (c) about sequential completeness are proved in \cite[Proposition 2.7]{DT1}. The remaining part of (c)
(concerning total disconnectedness, countable compactness and $\omega$-boundedness) is  clear. 

To check (b) denote by $q: K \to  \widetilde {G}$ the quotient map. Since $q(td_p(K) = td_p( \widetilde {G})$ (\cite{DPS}), 
strong $p$-density of $G_1$ provides $k\in \N$ with $p^ktd_p(K) \leq G_1$, which obviously implies strong $p$-density of $G$. 
If $G$ is strongly $p$-dense, say $p^ktd_p(\widetilde {G}) \leq G$, then for $x\in td_p(G_1)$ one has 
$q(x) \in td_p(\widetilde {G})$, so $p^kq(x) = q(p^kx) \in G$. Therefore, $p^k x \in G_1$.  
\end{proof}

\begin{corollary}\label{coro:tor_free} If $G$ is a \sco\ minimal   Abelian group with $c(G)=c(\widetilde G)$, then  $G/c(G)$ is \sco\ and minimal.
\end{corollary}

\begin{proof}
By Proposition \ref{tor_free} there exists a compact torsion-free Abelian group $K$ with $w(K)=w(G)$, a dense \sco \ minimal subgroup $G_1$ of $K$ and a 
compact totally disconnected subgroup $N$ of $G_1$ such that $G_1/N\cong G$. 
Let $f:K\to K/N\cong \widetilde {G}$ be the canonical projection. Since $f$ is surjective, we have $f(c(K))=c(\widetilde {G})$.  
Since $c(G)=c(\widetilde G)$, $G_1$ contains $c(K)$. Moreover, as $K$ is torsion-free, one has $K=c(K)\times D$, where $D$ is a compact totally disconnected subgroup of $K$.  Then $G_1=c(K)\times G_2$, where $G_2=D\cap G_1$ is a closed subgroup of $G_1$. Hence
$G_2$ is \sco\ and minimal. As $G/c(G) \cong G_1/c(K) = (c(K)\times G_2)/c(K) \cong G_2$, we  deduce that $G/c(G)$ is \sco\ and minimal.
\end{proof}

Let us note that as far as only minimality of $G/c(G)$ is concerned, Corollary \ref{coro:3space}
applies, so that only compactness of $c(G)$ suffices to this end (so the equality $c(G)=c(\widetilde G)$ is not needed). 

\subsection{Proof of Theorems A, B, C, D and Corollary \ref{Crit_prod} }\label{proofs}

The major step towards the proof of Theorem A  concerns the totally disconnected case. The following  proposition and its immediate corollary go in this direction:

\begin{proposition}\label{VU}
A  totally disconnected \sco\  \mi \  Abelian group is strongly $p$-dense for every prime $p$. 
\end{proposition}

Applying Theorem \ref{Stoy}, one obtains: 

\begin{corollary}\label{beta(td)}    
All powers of a  totally disconnected \sco\  \mi \  Abelian group are \mi. \end{corollary}

We split the proof of Proposition \ref{VU} in several steps. 

\smallskip 
 
 The next lemma is a local version of Proposition \ref{VU} and provides an essential tool for the proof of Theorem C(a) as well. In the case of  subgroups of $\Zp^\alpha$
  (which is the key to the proof in the general case of the lemma), it was  proved in \cite[Main Lemma]{DTo} for $\omega$-bounded groups $G$ by means of a
specific functorial correspondence, based on Pontryagin duality, between precompact groups covered by their compact subgroups and
linearly topologized groups introduced in \cite{T1}. This duality
technique essentially  used  the assumption of $\omega$-boundedness and cannot work even
under  the assumption of countable compactness. The proof given here in the \sco\ case is completely different, it has purely topological nature.  

\begin{lemma}\label{coroVU}
 Let $p$ be a prime number.  If $G$ is a dense  \sco\ minimal subgroup $G$ of a pro-$p$-group $A$, then 
there exists $k\in \N$ such that $p^kA\sq G$.\end{lemma}

\begin{proof} We first show that the proof can be reduced to the case when $A ={\Z}^{\alpha }_p$. By Proposition \ref{tor_free}, 
$A$ is a quotient of $B={\Z}^{\alpha }_p$, with $\alpha= w(A)$ and $G_1 = f^{-1}(G)$ is a sequentially complete minimal subgroup of $B$.
If there exists $k\in \N^*$ such that $p^k B \sq G_1$, then clearly, $p^kA\sq G$. 

Obviously, $A={\Z}^{\alpha }_p$ is a $\Z_p$-module. Moreover, since for $a\in G$ 
the cyclic $\Z_p$-submodule $\Z_pa \cong \Z_p$ is compact metrizable, so
it is contained in the \sco \ subgroup $G$. Hence, $G$ is a submodule of $A$. 
By the minimality of $G$ and Theorem \ref{Tot_Min_Crit}, for $0\ne x\in A$ the cyclic $\Z_p$-submodule $\Z_px = \overline{\langle x\rangle }$ non-trivially meets $G$, hence
$A/G$ is torsion $\Z_p$-module, so a $p$-group. We need to prove that it is a {\em bounded} $p$-group. 
Clearly, it is enough to prove that for every sequence $\{x_k: k\in 
\N^*\}\sq A$ there exists $n\in \N^*$ (independent on $k$) such that $p^nx_k\in  G$ for all $k$.

Consider the group $\Z_p^\omega$ equipped with the $p$-adic topology  $\tau_p$. A base of the neighbourhoods of 0
in the topology $\tau_p$ is given by is given by the  open  subgroups  $V_n=p^n\Z_p^\omega$, $n\in \N^*$. Then $\Z_p^\omega$ is a complete metrizable topological group.
Let $E\subseteq \Z_p^\omega$ be the set of all sequences $\{a_k: k\in \N^*\}$ with $\lim a_k=0$.\\

\begin{claim}
$E$ is a closed pure submodule of $\Z_p^\omega$, i.e., $E \cap p^m  \Z_p^\omega = p^mE$ for all $m\in \N^*$.
\end{claim}

\begin{proof} Obviously, $E$ is a submodule of $\Z_p^\omega$. To check that $E$ is closed, pick $c = (c_k) \in \Z_p^\omega \setminus E$. Then $c_k \not \to 0$, so there exists $m$ such that for every $k\in \N^*$ there exists $l> k$ such that $c_l \not \in p^m\Z_p$. Then $(c+ V_m )\cap E = \emptyset$. 
Indeed, if $a = (a_k) \in E$, then there exists $t>k$ such that $a_s \in p^m\Z_p$ for all $s\geq t$. Taking now $s\geq t$ 
such that  $c_s \not \in V_m$ we have $a_s - c_s \not \in p^m\Z_p$, as $a_s \in p^m\Z_p$. 
Hence $a-c \not \in V_m$ and $a \not \in c+V_m$. This proves that $(c+ V_m )\cap E = \emptyset$. 

It remains to check that $E \cap p^m  \Z_p^\omega = p^mE$ for all $m\in \N^*$.
To check the non-trivial inclusion $E \cap p^m  \Z_p^\omega \leq p^mE$ for $m\in \N^*$ pick 
$a=(a_k) \in E \cap p^m  \Z_p^\omega$. Then $a_k= p^mb_k$ for some $b_k \in \Z_p$ and for all $k$. 
Since obviously $b_k \to 0$ in $\Z_p$, $b = (b_k) \in E$, so $a \in p^mb\in p^mE$. \end{proof}

Now we define a map $\varphi :E\to A$ putting $\varphi (e_n)= x_n$, where $\{e_n: n \in \N^*\}$ is the canonical base of the dense submodule $E'=\Z_p^{(\omega)}$
of $E$. This assignment can be uniquely extended to a $\Z_p$-homomorphism $\varphi :E'\to A$ which is obviously continuous, so can be further (uniquely) extended to a  continuous $\Z_p$-homomorphism $\varphi :E\to A$ satisfying  $\{a_k\}\buildrel {\varphi} \over \mapsto \sum a_kx_k$. 

\smallskip

The $\Z_p$-submodule $F=\varphi^{-1}(G)$ of $E$ is sequentially closed (and hence closed) and $E/F$, being isomorphic to a submodule of $A/G$, is a $p$-group. 
Therefore,  for every $s\in E$ there exists $n\in \N^*$ such that $p^n s \in F$, so $E/F =\bigcup_{n\in \N^*}(E/F)[p^n]$ and 
$(E/F)[p^n]$ is closed in $E/F$ equipped with the quotient topology. Let $q: E \to E/F$ be the canonical homomorphism.
Then $E_n= q^{-1}((E/F)[p^n])$ is closed in $E$ and $E=\bigcup_{n\in \N^*}E_n$.  By the Baire category theorem implies $E_k$ has a non-empty interior for some $k\in \N^*$. Since  $E_k$ is a subgroup, this yields that $F_k$ is open. Hence there exists an open subgroup  $E\cap V_m = p^m E $ in $E$ such that $p^k(p^m E ) = p^{m+k} E\sq F$ for some $k\in \N^*$. This means, that for every sequence $\{a_i\}\sq p^{m+k}  E$ one has $\varphi(\{a_i\})\in G$. In particular, for every $i\in \N^*$ we have $p^{m+k} x_i = \varphi(p^{m+k} e_i) \in G$. 
\end{proof}

 \bigskip
 
 Now we prove first Theorem C that will be used below in the proof of Proposition \ref{VU} and Theorem B. 

 \bigskip

\noindent {\bf Proof of Theorem C.} Both (a) and (b) are proved in \cite[Theorem 3.6]{DT1} in full detail although the proof is quite long and heavy.
Furthermore, (a) and (b) are not articulated, with the blanket assumption that $w(c(\widetilde{G}))$ is not Ulam measurable.  We preferred to split (a) and (b) so that it becomes clear where the relevant assumption that $w(c(\widetilde{G}))$ is not Ulam measurable is used (only in item(b)), as well as the crucial role of Lemma \ref{coroVU} in the proof of (a). 

(a) The proof of the fact that $c(G)=c(\widetilde{G})$ whenever $c(G)$ is compact can be found in 
Steps 1 and 3 of the proof of \cite[Theorem 3.6]{DT1}. Both steps essentially use Lemma \ref{coroVU}, in the particular (but essential) case
when $A$ is a power of $\Z_p$. 

We prefer to point out a self-contained argument for deducing the remaining part of (a) from the equality $c(G)=c(\widetilde{G})$, rather than refer to \cite{DT1}. Indeed, Corollary \ref{coro:tor_free} implies that this equality yields minimality and sequential completeness of  $G/c(G)$. 

(b) The proof of this item is Step 2 in the proof of \cite[Theorem 3.6]{DT1}. $\Box$.

\bigskip

\noindent {\bf Proof of Proposition \ref{VU}.} We have to prove that a
 totally disconnected \sco\  \mi \  Abelian group $G$ is strongly $p$-dense for every prime $p$. 

By Theorems \ref{Prec:Thm} and C, the completion $K$ of $G$ is compact and totally disconnected, hence
$K = \prod_p K_p$, where  $K_p=td_p(K)$ is a pro-$p$-group, according to Fact \ref{prop_td_p}(3).  
By Theorem \ref{Products},  $G = \prod_p G_p$, where $G_p=td_p(G) = G \cap K_p$ is a dense essential subgroup of $K_p$. 
Moreover, $G_p$ is \sco, being a closed subgroup of $G$. By Lemma \ref{coroVU}, there exists $k_p\in \N^*$ with $p^{k_p}K_p \sq G_p \sq G$. 
$\Box$

\bigskip

\noindent{\bf Proof of Theorem A.}  We have to prove that for a minimal \sco \  Abelian group $G$ the following are equivalent:
\begin{itemize}   
     \item[(a)] all powers of $G$ are minimal;
     \item[(b)] $G^\omega$ is minimal.  
     \item[(c)] $G$ contains $c(\widetilde{G})$;  hence, $c(G)=c(\widetilde{G})$ is compact. 
\end{itemize}

\medskip 
 
To prove the implication (b) $\to$ (a) assume that $G^\omega$ is minimal. Since for every prime $p$ every closed monothetic subgroup of $td_p(G)$ is compact, it follows from \cite{D5} (see also \cite[Corollary 6.2.8]{DPS}) that all powers of $G$ are minimal, i.e., (a) holds.

To prove the the implication (a) $\to$ (c) assume that all powers $G^\alpha$ are minimal. Then by Theorem \ref{Stoy} $G$ is strongly $p$-dense for some prime $p$. By Lemma \ref{Str_d}, $G$  contains  $K = c(\widetilde{G})$. Since $K$ is divisible, this yields $td_p(K)$ is divisible too. Hence $G$ contains $td_p(K)$.  According to \cite[Theorem 2.9(a)]{DT1} $td_p(K)$ is sequentially dense in $K$. Therefore, the sequential completeness of $G$ implies $K\subseteq G$. In particular, $c(G)=K$ is compact.  

For the proof of the implication (c) $\to$ (b) assume that $c(G)=c(\widetilde G)$ holds true. By
Corollary \ref{coro:tor_free}, $G/c(G)$ is minimal and \sco. Since $G/c(G)$ is totally disconnected, Corollary \ref{beta(td)} yields that $(G/c(G))^\omega$ is \mi. Since $(G/c(G))^\omega\cong G^\omega/(c(G))^\omega$ and $c(G)^\omega$ is compact, we deduce that $G^\omega$ is minimal, by Proposition \ref{3space}.
 $\Box$
 
 \bigskip
 
\noindent {\bf Proof of Corollary \ref{Crit_prod}.} We have to prove that if $\{G_i\}_{i\in I}$ is a family of \sco \ Abelian groups,  then $G=\prod_{i\in I}G_i$  is minimal if and only if  every countable subproduct is minimal and in such a case there exists a co-finite subset $J\subseteq I$ such that all groups $G_i^\omega$, $i\in J$,
are \mi. 

The first part follows from \cite[Theorem 6.2.7]{DPS}. According to the same theorem, there exists a co-finite subset $J\subseteq I$ such that 
all groups $G_i$, $i\in J$, are strongly $p$-dense. According to Theorem \ref{Stoy} this yields minimality of $G_i^\omega$. $\Box$

 \bigskip
 
\noindent {\bf Proof of Theorem B.} The implication (c) $\to$ (b) is trivial. The implication (b) $\to$  (a) follows from Theorem C. 
 
 (a) $\to $ (c) Let  $\alpha$ be  Ulam measurable and let  $I$ be a set of cardinality $\alpha$.  Since the cardinal $\alpha$ is Ulam measurable, there exists a
countably complete ultrafilter ${\mathcal  F}$ on $I$. Following \cite[Proposition 3.6.1]{DPS}, in the sequel we use 
the following properties of $\K$ and related notation. Taking the dual of the short exact sequence $0 \to \Z \to \Q \to \Q/\Z \to 0$
one obtains a  short exact sequence $0 \to \widehat{\Q/\Z}   \to \K = \widehat{\Q} \to  \widehat{\Z} = \T \to 0$. Here we denote by
$\HH$ the closed subgroup $\widehat{\Q/\Z} \cong \prod_p\Z_p$ so $\K$, so that $\HH_p:= td_p(\HH) \cong \Z_p$, up to isomorphism. 

\smallskip 

The subgroup $B_{\mathcal  F}=\bigcup_{S\in{\mathcal  F}}\K^{I\setminus S}$ of $\K^I$ is dense, since $\K^{(I)}\leq B_{\mathcal  F}$ and the former subgroup is dense.
Since the filter ${\mathcal  F}$ is countably complete, the family of subgroup $\{\K^{I\setminus S}\}_{S\in{\mathcal  F}}$ is $\omega$-directed, 
hence the subgroup $B_{\mathcal  F}$ is $\omega$-bounded. Then the group $G_{\mathcal  F}:= \HH^I + B_{\mathcal  F}$  is {\bf  $\omega$-bounded} as well, as $\HH^I$ is compact.
So far we proved that $G_{\mathcal  F}$ is a dense $\omega$-bounded subgroup of $\K^I$. To see that it is {\bf not compact}, it suffices 
to note that it is a {proper} subgroup of $\K^{I}$ (indeed, if $a\in \K\setminus \HH$, then the  point ${\bf a}$ of $\K^{I}$ 
 with all coordinates equal to $a$ does not belong to  $G_{\mathcal  F}$). 

\smallskip 

It only remains to check that $G_{\mathcal  F}$ is {\bf minimal}. According to Theorem \ref{Tot_Min_Crit}, 
we have to prove that $G_{\mathcal  F}$ is {essential} in $\widetilde{G}_{\mathcal  F}=\K^{I}$. 
By Lemma \ref{loc:criterion}, it suffices to check that  for every prime $p$ and for every $x=(x_i)\in 
td_p(\K^{I}) =( td_p(\K))^{I}$ there exists $n\in \N^*$ such that $p^nx\in G_{\mathcal  F}$. Let $supp(x)=\{i\in I:x_i\ne 0 \}$. 
For $n\in \N^*$ set $T_n=\{i\in supp (x): p^n x_i \not \in  \HH_p \}$.  Obviously $T_1\supseteq  T_2\supseteq \ldots $
is a chain in  $X$ and $\bigcap T_n=\emptyset$ as $i\in supp (x)$
yields  there exists $k$, so that $i\not \in  T_{k}$, as $x_i\in  td_p(\K) = \bigcup_{k=0}^\infty p^{-k}\HH_p$. Hence 
\begin{equation}\label{LastEq}
I= \bigcup_{n=1}^\infty I\setminus T_n.
\end{equation}
 Since ${\mathcal  F}$ is a countably complete ultrafilter, it follows from (\ref{LastEq}) that   $I\setminus T_{n_0}\in {\mathcal  F}$ for some $n_0\in \N^*$.
Hence $\K^{T_{n_0}}\subseteq B_{\mathcal  F} \subseteq G_{\mathcal  F}$. Split $x=x'+x''$, such that $supp(x')\cap supp (x'')=\emptyset$ and $supp(x')\subseteq T_{n_0}$. Then 
$supp(x'')\subseteq I\setminus T_{n_0}$ so that $p^{n_0}x''\in {\HH^I}\subseteq G_{\mathcal  F}$, while $x'\in 
\K^{T_{n_0}}\subseteq G_{\mathcal  F}$. Consequently,  $p^{n_0}x\in G_{\mathcal  F}$.

\smallskip 

To prove the last assertion of the theorem assume now that $\alpha$ is measurable. Then we can choose the filter ${\mathcal  F}$ complete, i.e., stable under intersections of families of cardinality $< \alpha$. As above, the family of subgroup $\{\K^{I\setminus S}\}_{S\in{\mathcal  F}}$ is $\beta$-directed
for every $\beta < \alpha$, hence the subgroup $B_{\mathcal  F}=\bigcup_{S\in{\mathcal  F}}\K^{I\setminus S}$ of $\K^I$, defined as above, is $\beta$-bounded for every $\beta < \alpha$. Then also the group $G_{\mathcal  F} = \HH^I + B_{\mathcal  F}$  is {\bf   $\beta$-bounded for every $\beta < \alpha$}. The rest of the proof works in the same way. 
 $\Box$
 
 \bigskip

The following theorem provides a proof of  a stronger version of Theorem D: 

 \begin{theorem}\label{TD_cor} 
  Let $G$ be a \sco\  Abelian such that $G^\omega$ is \mi.  Then then for every $p\in\Prm$ there exists a $k_p\in \N^*$ such that
$G$ contains the compact subgroup $N= \overline{\langle p^{k_p}td_p(\widetilde{G}):p\in \Prm\rangle} $
(the closure  taken in $\widetilde{G}$)  and $G/N\cong \prod_{p \in \Prm} B_p$, where $B_p$ is a bounded $p$-torsion minimal group for
every prime $p$. 
\end{theorem} 

\begin{proof} By Theorem A, $c(G)=c(\widetilde{G})$ is compact.  By Corollary \ref{coro:tor_free},  $G_1:= G/c(G)$ is minimal and \sco. 
Put $A = \widetilde{G}/c(G)$ and let $q: \widetilde{G} \to \widetilde{G}/c(G)$ be the canonical homomorphism. 
Clearly $A$ is the completion of $G_1$.  Put for brevity $A_p = td_p(A)$. 
By Fact \ref{prop_td_p}(3), 
 $ A = \prod_p A_p$, since $A$ is totally disconnected. Proposition \ref{VU} implies that for every prime $p$ there exists $k_p\in \N^*$ such that 
\begin{equation}\label{Eq1}
N_p:=p^{k_p}A_p\leq  G_p
\end{equation}
 Hence, $S_1=\bigoplus_{p\in \Prm}pN_p\leq G_1$. Since $S_1$ is sequentially dense in the closed subgroup $N_1= \prod_{p \in \Prm} pN_p$ of $A$
and $G_1$ is \sco, we deduce that $N_1 \leq G_1$. Hence, $N = q^{-1}(N_1)\leq G$. 
By Fact \ref{prop_td_p}, $q(td_p(\widetilde{G})) = A_p$ for every prime $p$, hence 
the subgroup $S=\langle p^{k_p+1}td_p(\widetilde{G}):p\in \Prm\rangle$ of $\widetilde{G}$ satisfies $q(S) = S_1 \leq  N_1$, so $S\leq N\leq G$. Hence, 
$S \leq L:=  \overline{S} \leq N$ and $S_1= q(S) \leq q(L)\leq q(N) = N_1 = \overline{S_1}$. Since
$q(L)$ is compact, we deduce that $q(L) = q(N)$. As $\ker q  = c(\widetilde{G}) \leq L \leq N$ by (the proof of) Lemma \ref{Str_d}, 
we conclude that $L = N$. This proves the first assertion of the theorem. 

To prove the second one we put $B_p = A_p/pN_p = A_p/p^{k_p+1}A_p$ and note that $A/N_1 \cong  \prod_{p \in \Prm}B_p$. 
Since $G/N \cong G_1/N_1$, it is enough to prove that $G_1/N_1$ is \mi. According to Theorem \ref{Tot_Min_Crit}, 
this amounts to check that $G_1/N_1$ is a an essential subgroup of $A/N_1$. Since  for every prime $p$ the group $B_p$ is a bounded $p$-group, it suffices to check that 
$Soc(B_p) \leq G_1/N_1$, according to Lemma \ref{loc:criterion}. 

The minimality of $G_1$ and Lemma \ref{loc:criterion} yield that $Soc(A) = \bigoplus_{p \in \Prm} A[p]\leq G_1$. By virtue of (\ref{Eq1}), we obtain 
\begin{equation}\label{Eq2} 
  A[p] + p^{k_p}A_p \leq G_1. 
 \end{equation}
To check that $Soc(B_p) \leq G_1/N_1$ pick 
 $\overline{a} = a + pN_p \in Soc(B_p)$, then $p\overline{a} =0$ in $B_p$, so $pa \in pN_p$. Consequently, $pa = p^{k_p+1}b$ for some $b \in A_p$. Hence, $t:=a- p^{k_p}b \in A[p]$. Now (\ref{Eq2}) gives $a= t + p^{k_p}b\in G$. Therefore, $\overline{a} \in G_1/N_1$. \end{proof} 

\section{Final comments and questions}

The groups considered in this section are not necessarily Abelian. 

Sequential completeness is not preserved under taking  continuous homomorphic images.  Following \cite{DT3}, 
call a group $G$ {\it  $h$-\sco} if all  continuous homomorphic images of $G$ are \sco. Countably compact groups are obviously $h$-sequentially complete.
 Nilpotent sequentially $h$-complete groups are precompact \cite[Theorem 3.6]{DT3}, while 
 any precompact sequentially $h$-complete group is pseudocompact \cite[Theorem 3.9]{DT3}. 

\begin{question}\label{Diag}
\begin{itemize}
    \item[(1)] Find an example of a precompact $h$-\sco \  group that is not \cc.
    \item[(2)] Find an example of a pseudocompact \sco \  group that is not $h$-\sco.
    \item[(3)] Find an example of a (precompact) \sco \  group that has
non-sequentially-complete quotients.
   \item[(4)] Are precompact \sco \  \mi \ groups pseudocompact? 
   \item[(5)] Are  \sco \  \mi \ totally (hereditarily) disconnected groups zero-dimensional? 
\end{itemize}          \end{question}

In some cases (4) has a positive answer -- when the group has large hereditarily pseudocompact (in particular, \cc) subgroups.
In view of \cite[Theorem 1.2]{D8} (5) is true for $h$-\sco \ groups.  On the other hand, it was shown in \cite{DT1} that (5) is true for  \sco \  \mi \ Abelian groups. 

\begin{question}\label{Ques_conv_seq} Which of the following properties of an infinite topological group guarantee the existence of non-trivial converging sequences:
i) \cc \  and minimal; ii) \cc \ and \tm; iii) \tm?  \end{question}

The following question was left open in \cite{DT1}: 

\begin{question}
Let $G$ be a \sco \ \mi \ Abelian group. Is then $G/c(G)$ \sco? 
\end{question}

Finally, we recall an old open problem raised by the second named author closely related to the topic of this paper. We now cancel our assumption that all groups under consideration are Abelian.  

\begin{question}\label{QuesUsp1} {\rm (Uspenskij \cite{Usp})} Is it true that arbitrary products of complete minimal groups are minimal? 
What about infinite powers?
\end{question}

Note that the answer is positive in the Abelian case, since minimal complete Abelian groups are compact, by Theorem \ref{Prec:Thm}. ``Complete" in Question~\ref{QuesUsp1} means complete with  respect to the two-sided uniformity. The question also makes sense for the case of Weil-complete
groups, that is, groups complete with respect to the left (or right) uniformity. A particular instance of such groups are the locally compact ones. The above question
seems to be open even in the case of locally compact minimal groups (many instances of positive answer in this case can be found in \cite{RS}, yet no 
general result seems to be available to the best of our knowledge).  

\bibliographystyle{amsplain}

\end{document}